  \documentclass[12pt,letterpaper,oneside]{amsart}

  \usepackage{amsmath,amssymb,amsthm}
  \usepackage{geometry}
  \usepackage{accents}
  \usepackage{graphicx}
  \usepackage{enumerate}
  \usepackage{color}
  \usepackage{xspace}

  \usepackage{epsfig} 

  \usepackage{hyperref}

  \newcommand{\calC}{\mathcal{C}}
  
  \newcommand{\calE}{\mathcal{E}}

  \newcommand{\calV}{\mathcal{V}}

  \newcommand{\CC}{\mathbb{C}}

  \newcommand{\RR}{\mathbb{R}}
  \renewcommand{\SS}{\mathbb{S}}

  \newtheorem{theorem}{Theorem}[section]
  \newtheorem{proposition}[theorem]{Proposition}

  \newtheorem{question}[theorem]{Question}
   \newtheorem*{theoremA}{Theorem A}
 \newtheorem*{theoremB}{Theorem B}

  \theoremstyle{definition}
  \newtheorem{definition}[theorem]{Definition}
  
  \newtheorem*{claim*}{Claim}
  
  \newtheorem{example}[theorem]{Example}
  
  \newtheorem*{question*}{Question}
  \newtheorem*{answer*}{Answer}
  \newtheorem*{application*}{Application}

  \theoremstyle{remark}
  \newtheorem{remark}[theorem]{Remark}
  \newtheorem*{remark*}{Remark}

  \newcommand{\secref}[1]{\S\ref{Sec:#1}}

  \newcommand{\exaref}[1]{Example~\ref{Exa:#1}}

  \newcommand{\Aut}{\ensuremath{\operatorname{Aut}}\xspace} 
   
  \newcommand{\sE}{{\sf E}}
  \newcommand{\sP}{{\sf P}}

  \newcommand{\sB}{{\sf B}}  
  \newcommand{\sM}{{\sf M}}  
  \newcommand{\sA}{{\sf A}}   
  \newcommand{\sT}{{\sf T}}     
  \newcommand{\sV}{{\sf V}}

  \newcommand{\param}{{\mathchoice{\mkern1mu\mbox{\raise2.2pt\hbox{$
  \centerdot$}}
  \mkern1mu}{\mkern1mu\mbox{\raise2.2pt\hbox{$\centerdot$}}\mkern1mu}{
  \mkern1.5mu\centerdot\mkern1.5mu}{\mkern1.5mu\centerdot\mkern1.5mu}}}

  \newcommand{\from}{\colon\thinspace} 

   \newcommand{\al}{\alpha}
  \newcommand{\lam}{\lambda}
  \newcommand{\ulam}{\underline{\lambda}}  
   
  \newcommand{\eps}{\varepsilon}
   
    \newcommand{\glue}{\mu}
    
  \newcommand{\out}{\text{Out}(F_n)}
   \renewcommand{\top}{\textup{top}}
   \renewcommand{\bot}{\textup{bot}}

   \newcommand{\md}{\textup{ mod}}
   \newcommand{\dout}{d_\text{out} }
  \newcommand{\din}{d_\text{in}}
            
    \newcommand{\xfp}{ \accentset{\circ}{X}} 
    \newcommand{\gfp}{\accentset{\circ}{\glue}}
     \newcommand{\HR}{ \textup{HE}} 
     \newcommand{\VR}{ \textup{VE}} 

\begin{document}


  \title    {Mapping tori of small dilatation irreducible train-track maps}
  \author   {Yael Algom-Kfir and Kasra Rafi}
  
  \date{\today}

  \begin{abstract} 
An irreducible train-track map on a graph of rank $n$ is $\sP$--small if its dilatation 
is bounded above by $\sqrt[n]{\sP}$.  We prove that for every $\sP$ there 
is a finite list of mapping tori $X_1, \ldots, X_\sA$, with $\sA$ depending only 
on $\sP$ and not $n$, so that the mapping torus associated to 
every $\sP$--small irreducible train-track map can be obtained by surgery 
on some $X_i$. We also show that, given an integer 
$\sP>0$, there is a bound $\sM$ depending only on $\sP$ and not $n$, 
so that the fundamental group of the mapping torus of any $\sP$--small irreducible 
train-track map has a presentation with less than $\sM$ generators and $\sM$ 
relations.  
\end{abstract}
  
  \maketitle
  

\section{Introduction} \label{Sec:Intro}

Given a space $X$ and a continuous map $f \from X \to X$ the mapping torus of $f$ is the space $M_f = X \times [0,1]/\sim$ where $(x,0) \sim (f(x),1)$. By Van Kampen's theorem, this topological construction corresponds to an HNN extension of the fundamental group. Explicitly, if $\Gamma$ is the fundamental group of $X$ and $\phi=f_* \from \Gamma \to \Gamma$ is the induced homomorphism then the fundamental group of $M_f$ is \[ \Gamma_\phi := \Gamma *_\phi = \langle \Gamma, t \mid t \gamma t^{-1} = \phi(\gamma) \text{ for } \gamma \in \Gamma \rangle \] When $\phi$ is an automorphism, $\Gamma_\phi$ is a semi-direct product.\\

The study of mapping tori of automorphisms of surface groups and free groups has attracted a great deal of interest. It is known, for example, that these groups have a quadratic Dehn function. In the surface case this was proven by Epstein and Thurston \cite{automatic} and for free groups it was shown by Bridson and Groves \cite{BriGroves}. Thurston proved that if $\phi$ is a pseudo-Anosov mapping class then $M_\phi$ admits a hyperbolic metric. Therefore, when the surface is closed $\Gamma_\phi$ is a uniform lattice in $\text{PSL}(2,\CC)$, rendering all pseudo-Anosov mapping tori quasi-isometric. The situation for automorphisms of the free group is more complicated. Work of Brinkman \cite{Brink}, and Bestvina and Feighn \cite{BFcombination} implies that when $\phi$ does not fix a conjugacy class i.e. when it is \emph{atoroidal}, then $\Gamma_\phi$ is Gromov hyperbolic. However, not all atoroidal mapping tori are quasi-isometric to each other. Bowditch \cite{Bow} proved that the Gromov boundary of $\Gamma_\phi$ contains local cut points if and only if $\phi$ preserves a free splitting of $F_n$. Some atoroidal automorphisms preserve a free splitting and some do not, therefore, their mapping tori will not be quasi-isometric to each other. In this paper we study mapping tori of a family of automorphisms, which we call itt automorphisms, via their 2-cell structure induced by particularly nice graph maps. \\

Bestvina and Handel \cite{BH} defined the notion of a train-track map as a kind of normal form for an irreducible outer automorphism (see section \secref{background} for definitions). They proved that any irreducible outer automorphism can be represented by an irreducible train-track map. However, the induced automorphism of an irreducible train-track map need not be irreducible as an automorphism of $F_n$ (see a recent preprint by Kapovich \cite{Kap} for the exact distinction). We call $\phi \in \out$ an itt (irreducible train-track) automorphism if it can be represented by an irreducible train-track map. To an irreducible train-track map $f \from G \to G$ we can assign a real number $\lam_f \geq 1$. If $w \in F_n$ is not $\phi$-periodic then $\lam_f$ is equal to the exponant of the growth rate of $k \to\phi^k(w)$. Thus, $\lam_f$ does not depend on the train-track map $f$ representing $\phi$ (which is not unique). We denote it by $\lam_\phi$ and call it the dilatation of $\phi$.  We say $\phi$ is $\sP$--
small if 
$$
\lambda_\phi \leq \sqrt[n]{\sP}. 
$$
where $n = \text{rank}(\pi_1G)$. Every itt automorphism is $\sP$--small for a large enough $\sP$. Consider 
\[ 
\ulam(n) = \inf\{\log \lam_\phi \mid \phi \text{ is an itt automorphism in } \out \}. 
\]
In sections \secref{Example} and \secref{Lower} we prove that 
\[ \frac{\log 3}{3n-3} \leq \ulam(n) < \frac{\log 2}{n}. \]
In \secref{Example} we give examples of train-track maps whose dilatation is marginally smaller than the upper bound. In \secref{Lower} we give an argument for the lower bound. For comparison, if $S$ is a closed surface of negative Euler characteristic $\chi(S)$ denote \[\underline{\mu}(n) = \inf\{\log \lam_\phi \mid \phi \text{ is a pseudo-Anosov map on }S \text{ with } |\chi(S)| = n \} \] then
\[ \frac{\log 2}{6 n} \leq 
\underline{\mu}(n) \leq 
\frac{2\log\left(\frac{3+\sqrt{5}}{2}\right)}{n}. 
\] 
The upper bound was proved by Aaber-Dunfield and Hironaka \cite{AD, Hiro} and the lower bound by Penner \cite{Pen}. This is false when the surface is not closed. For $S = S_{g,p}$ with $g \geq 2$ constant and $p\geq1$ Tsai \cite{Tsai} proved that the minimal dilatation is on the order of $\frac{\log p}{p}$. \\

Our first theorem is an analogue of the following theorem of Farb, Leininger and Margalit \\

\begin{theorem}\cite{FLM, Agol} For each $\sP > 1$ there exist finitely many complete, noncompact, hyperbolic 3--manifolds $M_1, \dots , M_r$ fibering over $\SS^1$, with the property that any pseudo-Anosov homeomorphism of a surface $S$ with dilatation smaller than $\sP^{1/|\chi(S)|}$ occurs as the monodromy of some bundle obtained by Dehn filling one of the $M_i$ along boundary slopes of a fiber. 
\end{theorem}

\noindent We define mapping tori surgery and prove the following 

\begin{theoremA} \label{Thm:Small}
For every $\sP>1$, there is a finite set of $2$--complexes $X_1, \ldots , X_\sA$, which are mapping tori of self maps of graphs, so that the following holds. If $f \from G \to G$ is a $\sP$--small irreducible train-track map on a graph $G$, then $M_f$ is homeomorphic to a $2$--complex that is obtained by surgery on some $X_i$.   
\end{theoremA}

In particular, surgery does not alter the number of ``essential" 2-cells in the mapping torus and only changes the structure of the edges. As a result we are able to prove a universal boundedness result on presentations of these fundamental groups. 

\begin{theoremB} \label{Thm:BoundedPres}
There is a number $\sM$ depending only on $\sP$ (independent of $n$) so that 
if $\phi \in \out$ is a $\sP$--small  itt automorphism then $\Gamma_\phi$ has 
a presentation with at most $\sM$ generators and $\sM$ relations.
\end{theoremB}


\noindent\textbf{Acknowledgements.} The authors would like to thank PCMI for their hospitality. We express our gratitude to Catherine Pfaff and Lee Mosher for inspiring conversations.

\section{background}\label{Sec:background}

\subsection{Train-track maps and dilatation} 
\begin{definition}
A marked graph is a finite 1-complex $G$ together with an isomorphism $\tau:F_n \to \pi_1(G,*)$. We usually supress the isomorphism and denote it by $G$.
\end{definition}
Let $\phi$ be an automorphism of the free group $F_n$. 
\begin{definition}
Let  $f:G \to G$ be a map on a marked graph $G$ whose fundamental group is isomorphic to $F_n$ via $\tau$. The map $f$ is a topological representative of $\phi$ if
\begin{enumerate}
\item $f$ takes vertices to vertices, 
\item images of edges are immersed paths 
\item $\phi = \tau^{-1} \circ f_* \circ \tau$. 
\end{enumerate}
\end{definition}
Let $m$ be the number of edges in $G$. \emph{The transition matrix} 
$M= (a_{ij})$ associated to $f$ is an $m \times m$ matrix defined by  
\[ 
a_{ij} \text{ is the number of times }f(e_j)  \text{crosses}e_i \text{in either direction}
\] 
Observe that $M$ is a non-negative matrix. It is called 
\emph{Perron-Frobenius} if there is some $k$ so that all of the entries of $M^k$ are strictly positive. 
Let $\lam$ be the largest-modulus eigenvalue of $M$. Perron-Frobenius theory states that $\lam \geq 1$ is real. Recall further that if $\lam=1$ then $M$ is a permutation matrix and $f$ is a homeomorphism. The number $\lam$ is called the \emph{Perron-Frobenius  eigen-value} of $f$. \\

\begin{definition}
A topological representative $f$ of $\phi$ is a \emph{train-track map} if $f^k(e)$ is 
immersed for all edges $e \in G$ and all powers $k>0$. The map $f$ is an 
\emph{irreducible} train-track map if it is a train-track map, and its transition 
matrix is Perron-Frobenius.
\end{definition}

An equivalent description follows. Endow the graph $G$ with an orientation once and for all. A pair of directed edges $\{e,e' \}$ is called a \emph{turn} if $i(e) = i(e')$. Let $Df(e)$ denote the first edge in the path $f(e)$. A turn $\{e,e' \}$ is \emph{collapsable} if $Df(e) = Df(e')$. It follows from the work of Nielsen \cite{NielsenPres} and Stallings \cite{Stallings} that if $f$ is not a homeomorphism then it has a collapsable turn. The map $Df$ sends a turn to a turn. A turn $\{e,e'\}$ is \emph{illegal} if it is mapped to a collapsable turn under a positive iterate of $Df$, otherwise, it is \emph{legal}. An edge path is legal if it crosses no illegal turns. The proof of the following proposition is clear.

\begin{proposition} 
A map $f \from G \to G$ is a train-track map if and only if $f(e)$ is legal for every edge $e$. 
\end{proposition}

\begin{definition}
\emph{A trap} is an oriented connected graph with the property that every vertex has a unique edge initiating from it. Topologically, a trap is a directed circle glued to directed trees where each tree is glued to the circle at a single point and the edges of the tree are directed towards the circle, as if there was an attracting particle in the middle of the circle. 
\end{definition}

\begin{proposition}
$f \from G \to G$ is a train-track map iff $f^i(e)$ does not contain a backtracking segment (an edge followed by its inverse) for $1 \leq i \leq m$ where $m$ is the number of edges in $G$.
\end{proposition}
\begin{proof}
We must show that if a turn does not collapse after $m$ iterations of $f$ then it does not collapse at all, i.e. it is legal. In order to check whether or not a turn is legal we construct an auxiliary graph $\Delta_f$, called \emph{the derivative} of $f$. There is a vertex in $\Delta$ for each directed edge and an edge from $e$ to $e'$ if $Df(e) = e'$. 

The graph $\Delta_f$ is a union of disjoint traps. In order to check if a turn $\{e,e'\}$ is illegal we start with $v_e,v_{e'} \in \Delta_f$ corresponding to the directed edges $e,e'$ in $G$. We form the sequences  $a_i$ and $b_i$ of vertieces in $\Delta_f$ starting with $a_1=v_e, b_1 = v_{e'}$ and $a_i,b_i$ are the terminal vertices of the directed edges initiating at $a_{i-1},b_{i-1}$. The sequences $a_i$ and $b_i$ are like traps coming into orbit and rotating about the star. Eventually they move in a directed circle. $a_i$ and $b_i$ reach the cycle after no more than $m-1$ steps. If both $a_i$ and $b_i$ are on a directed circle and $a_i \neq b_i$ then they will never coincide. 
\end{proof}
\noindent This shows that deciding if a map is a train-track map is a finite check. \\

\begin{definition}\label{Def:naturalMetric}
Let $\overrightarrow{v}$ be a left eigen-vector of $M$ corresponding to the Perron-Frobenius eigen-value $\lam_f$ i.e. $v^t M = \lam v^t$. \emph{The natural metric} on $G$ induced by a train-track map $f$ is given by setting the length of the edge $e_i$ to be $v_i$ the $i$--th coordinate of $\overrightarrow{v}$. 
\end{definition}
The natural metric on $G$ is well defined up to multiplication by a scalar. When $G$ is endowed with the natural metric $len(f(e)) = \lam len(e)$ for every edge $e$ in $G$.

\begin{proposition}
If $\phi \in \Aut(F_n)$ is an itt automorphism represented by a train-track map $f \from G \to G$ then for any non-periodic $w$ in $\phi$, and for any basis $X$ in $F_n$: 
\begin{equation} \label{eq2}
\log\lam_f = \lim_{k \to \infty} \frac{\log|\phi^k(w)|_X}{k}
\end{equation}
Where $|w|_X$ denotes the word-length in the basis $X$, of the cyclically reduced word $w$. \\
In particular, $\lam_f$ is the same for all train-track representatives $f$ of $\phi$.
\end{proposition}

\begin{proof}
This proof is essentially written down in \cite{Tits0}, but we include it for completeness. Let $w$ be a conjugacy class of a word, it is represented in $G$ by an immersed loop $\al$. Endow $G$ with the natural metric. Consider the infinite sequence $\al_k = f^k_\#(\al)$. Since $f$ is a train-track map, the number of illegal turns in $\al_k$ is non-increasing. We consider the length of $\al_k$. If it is bounded, then $\al$ is preperiodic, since there are only finitely many conjugacy classes smaller than any given length. There is a threshold 
\[\sT = \frac{4 \text{BCC}(f)}{\lam -1}\]
so that if $\delta,\beta, \gamma$ immersed segments, $\beta$ is legal and $len(\beta)>\sT$ then $len(f^k_\#(\delta \cdot \beta \cdot \gamma)) \geq c \lam^k$. The reason is that there is a middle segment $\beta' \subseteq \beta$ that doesn't cancel at all and grows exponentially. If $\al$ is not preperiodic, then there is some $j$ so that $f^j_\#(\al)$ contains a legal segment of length $\sT$. Thus, $len(f^k_\#(\al)) \geq c \lam^k$ for $k \geq j$ (we absorbed $\lam^{-j}len(f^j(w))$ in $c$). The word length with respect to the basis $X$ of $F_n$, is  quasi-isometric to lengths of immersed paths in $G$. Thus, up to a multiplicative error $\displaystyle |\phi^k(w)|_X \asymp \lam^k$ hence the limit on the right of equation (\ref{eq2}) is $\log \lam$.  
\end{proof}

\subsection{Subdivision and un-subdivision of a 2-complex}
\begin{definition}
A \emph{2-complex} $X$ is a tuple $(\calV,\calE,\calC, \glue)$. The set $\calV$ is the set of vertices, $\calE$  is the set of edges, $\calC$ is the set of 2-cells that are polygonal subsets of $\RR^2$. The set $\glue$ is the collection of gluing maps. It is the union of the gluing maps of edges $\glue_\calE =  \{ \glue_e \mid e\in \calE \}$ and the gluing maps of cells $\glue_\calC = \{\glue_c \mid c\in \calC \}$. For every $e \in \calE$, $\glue_e: \partial e \to \calV$  is called a gluing map of $e$. The 1-skeleton of $X$ is the graph 
\[ \displaystyle X^{(1)} = \bigl( \bigcup_{e \in \calE} e \bigr) \cup_{\glue_{\calE}} \calV\] 
For $c \in \calC$, $\glue_c:\partial c \to X^{(1)}$ is  the gluing map of $c \in \calC$. We always assume that the gluing maps are linear on the edges of $c$, therefore the maps can be defined by specifying a directed path in $X^{(1)}$ for every edge of $\partial c$. The total space is 
\[ \displaystyle X =  \bigl( \bigcup_{c \in \calC} c  \bigr) \cup_{\glue} X^{(1)} \]
\end{definition}

\begin{definition}
A side of a 2-cell in $X$ is a pair $s=(c,\alpha)$ where $\alpha$ is an edge of $\partial c$ in $\RR^2$. 
\end{definition}

\begin{definition}
Suppose $e$ is an edge in  $X^{(1)}$ so that $\glue_\calC^{-1}(e)$ is the union of exactly two subintervals $J_1, J_2$ contained the sides $(\al_1,c_1)$ and $(\alpha_2,c_2)$ so that $c_1 \neq c_2$. We say that $e$ is a \emph{removable} edge. 
\end{definition}

\begin{definition}[Unsubdivision]\label{Defn:Unsub} If $X$ is a 2-complex and $e \in \calE$ is a removable edge we can glue the cells whose boundary contains  $J_1,J_2$ and remove $e$ from the set of edges. We will say that the new complex is obtained from $X$ by unsbdividing along $e$. 
\end{definition}

\begin{definition}[Subdivision]
Let $c$ be a 2-cell of $X$ so that $\partial c$ contains more than $3$ vertices. One may subdivide $c$ into two cells without changing the number of vertices by adding a diagonal $e$ between two non-adjacent vertices $v,w$ to the set $\calE$, and replacing $c \in \calC$ with two cells $c_1, c_2$ formed by subdividing $c$ along the diagonal from $v$ to $w$. We say that this complex is obtained from $X$ by subdividing $c$ along $(v,w)$.
%
\end{definition}

\subsection{Mapping tori}
Let $G$ be a graph. The mapping torus of a map $f\from G \to G$ is 
\[ M_f = G \times [0,1]/\sim \]
where the equivalence is generated by the relations $(x,0) \sim (f(x),1)$. We describe the 2-complex structure on $M_f$. 
\begin{enumerate}
	\item The set of vertices $\calV$ equals the set of vertices of $G$.
	\item The set of edges $\calE$ contains  the edges in $G$ which we call horizontal edges, and denote it by $\HR$. The edges in $\calE$ that are not horizontal are edges of the form $(v,f(v))$ for $v \in \calV$. We call these edges vertical edges and denote this set by $\VR$. We orient a vertical edge from $v$ to $f(v)$. 
	\item There is one square $c$ for every $e \in \HR$. We describe the gluing map of $c$. The top edge of $c$ is mapped to $e$, the left side $l$ of $c$ mapped to $\left(i(e), f(i(e)) \right)$ where $i(\cdot)$ denotes the initial vertex of that edge, and the right side $r$ of $c$ is glued to $\left(ter(e), f(ter(e)) \right)$, where $ter(\cdot)$ is the terminal vertex of that edge. The bottom edge edge is glued to the edge path $f(e)$. 
\end{enumerate}

\noindent It is straight-forward to check if a 2-complex structure is that of a mapping torus: 

\begin{proposition}\label{Prop:IdenMT}
The 2-complex $X$ is a mapping torus of a map $f \from G \to G$ iff:
\begin{enumerate}
\item The set of edges $\calE$ of $X$ may be divided into two sets: the set of vertical edges $\VR$ and the set of horizontal edges $\HR$. 
\item\label{Ugraph} The graph on the vertical edges defined by \[ U = \bigl( \bigcup_{e \in \VR} e \bigr) \cup_\glue \calV \] may be endowed with an orientation to make it a union of traps. Equivalently, there is a bijection $a \from \calV \to \VR$. 
\item There is a bijection $\top \from \calC \to \HR$. 
\item We define the horizontal graph by \[ G = \bigl( \bigcup_{e \in \HR} e \bigr) \cup_\glue \calV \]  we choose an orientation on $G$. For $c \in \calC$, the map $\glue_{c}$ is given by the path $\top(c)rw^{-1}l^{-1}$ where $r,l$ are positively oriented edges in $U$ and $w$ is an edge path in $G$. We denote the edge path $w$ by $\bot(c)$.
\end{enumerate} 
\end{proposition}
\begin{proof}
If $X$ satisfies items (1)-(4) we define the map $f \from G \to G$ as the linear map taking each vertex $v \in G$ to the terminal endpoint of $a(v)$ defined in item (2), and for $e$ an edge in $G$ we let $f(e) = \bot(c)$ where $c$ is the cell so that $\top(c) = e$. We check that $f$ is well defined: since $erw^{-1}l^{-1}$ is a connected path, $ter(r) = ter(w)$ and $ter(l) = ini(w)$, hence $f(ter(e))=ter(f(e))$ and $f(i(e))=i(f(e))$. The gluing maps given in (2),(4) are exactly those induced by $f$. 
\end{proof}

The fundamental group of the mapping torus $M_f$ is an HNN extension of $F_n$. Let $\phi = f_*$ then 
\[ \Gamma_f = F_n *_\phi = \langle x_1, \dots, x_n,t \mid tx_it^{-1} = \phi(x_i) \rangle \]
The universal cover $\widetilde{M}_f$ of $M_f$ is contractible. This follows from the fact that point inverses of the map from $\widetilde{M}_f$ to its Bass-Serre tree are contractible. 


\section{Surgery of mapping tori}

\begin{definition}
Let $f \from G \to G$ be a map, we call and edge
$e$  \emph{mixed} if its image is a concatenation of more than one edge. We call an edge $e$ \emph{dynamic} if it is contained in the image of a mixed edge, or if there is more than one edge that maps onto it. 
\end{definition}

%

\begin{definition}\label{Def:Qgraph}
We define an equivalence relation on the edges of $G$. The relation is generated by $e \sim f(e)$ if $f(e)$ is a single edge that is non-dynamic. An equivalence class of edges will be called a \emph{stack}. Similarly we define an equivalence relation between vertices of $G$, generated by: $v \sim w$ if $f(v)=w$. 
\end{definition}

Each stack has the form $\eps = \{ e, f(e), f^2(e), \dots , f^s(e) \}$ for $s \geq 0$ where $e$ is a dynamic edge and $f^s(e)$ is possibly a mixed edge. We call $f^s(e)$ \emph{the bottom edge} of the stack $\eps$. Note that $e$ is the only dynamic edge in $\eps$ and $f^s(e)$ is the only mixed edge in $\eps$ (if it is indeed mixed). The quotient of $G$ under these equivalences is a graph denoted $Q$, and the quotient map will be denoted $p \from G \to Q$. 

\begin{definition} \emph{The archetype} of $f$ is a map $f_Q: Q \to Q$ defined by as follows.  $f_Q$ fixes every vertex of $Q$, and $f_Q(\eps) = p(f(e))$ where $e$ is the bottom edge of $\eps$.
\end{definition}

\begin{remark}
The archetype of a train-track map of a fully irreducible automorphism need not be a train-track map, need not be irreducible and might not even be a homotopy equivalence. The most that can be said is that $f_Q:Q \to Q$ is a composition of a sequence of Stallings folds and a pinch map - a map that is a homeomorphism on the graph minus its vertices.   
\end{remark}

We define \emph{a surgery of a mapping torus}. Let $X$ be a mapping torus of the map $f \from G \to G$. Let $U$ the sub-graph of $X^{(1)}$ containing the vertical edges, as in Proposition \ref{Prop:IdenMT}(\ref{Ugraph}). A connected component of $U$ is an oriented graph on an equivalence class of vertices as in Definition \ref{Def:Qgraph}. \\

We want to remove $U$ from $X$, glue in a different graph $S$, and redivide the 2-cells so the output is again a mapping torus. We first unsubdivide the cells of $X$. For each non-dynamic edge, $e$ of $G$ there are exactly two 2-cells $c_e, c_{f^{-1}(e)}$ that are attached to $e$. Therefore $e$ is a removable edge and we may unsubdivide along $e$ as in definition \ref{Defn:Unsub}. By unsubdividing $X$ at along all the non-dynamic edges in a single stack $\eps$, one obtains a single 2-cell $R_\eps$ called the rectangle corresponding to $\eps$. The attaching map $\glue_{R_\eps}$ sends the top edge of $R_\eps$ to the dynamic edge $e$ in $\eps$, the bottom edge of $\eps$ to $f(e')$ for $e'$ the bottom edge of $\eps$. The right side of $R_\eps$ is sent to an edge path initiating at $ter(e)$ and terminating at $ter(f(e'))$ whose length is $|\eps|$, and the left side of $R_\eps$ is mapped similarly. \\

\begin{definition}
The 2-complex obtained from $X$ by unsubdividing at all the non-dynamic edges is called the \emph{the floor-plan} of $X$ and denoted $\xfp$.
\end{definition}

Let $K$ be the boundary of a small neighborhood of $U$ in $\xfp$. The vertices of $K$ correspond to equivalence classes of oriented dynamic edges in $G$. The equivalence is generated by the relations $e \sim E'$ if there is some mixed edge $e''$ such that $f(e'')$ contains $e'e$ as a subword. For each stack $\eps$ of $G$ there is an oriented edge in $K$, denoted $b_\eps$ from the dynamic edge $e$ to $Df(e')$ where $e'$ is the bottom edge of $\eps$. By unraveling the definitions we see that $K$ is a quotient of $\Delta_{f_Q}$ under the equivalence relation just described. 

\begin{definition}
The 2-complex $X_U$ obtained by \emph{cutting $X$ along $U$}, is the result of gluing $K$ to $\xfp -U$. Explicitly, we consider $\xfp - U$ and add the sides $s_\eps, s_{\bar\eps}$ to the rectangle-with-missing-sides $R_\eps$. The assignment $\chi(s_\eps) = b_\eps$ defines a map $\chi: X_U \to K$ so that $X_U = X_U \cup_\chi K$. 
\end{definition}

The glueing map $\gfp$ suppressed in the structure of $\xfp$ determines a map $\rho: K \to U$, so that $\xfp = X_U \cup_\rho U$. This implies the following proposition:


\begin{proposition}\label{Prop:SameFP}
Let $X$ be the mapping torus of $f$ and $Y$ the mapping torus of $f_Q$. Let $U$ be the graph of vertical edges in $X$,  $W$ the graph of the vertical edges in $Y$ then
\[ X_U = Y_W \]
\end{proposition}

We now wish to glue $X_U$ to a different graph $S$ via $\psi$. Let $|\calV_Q|$ be the number of vertices in $Q$. This number also equals the number of vertex orbits in $G$, and is the number of connected components of $U$.

\begin{definition}\label{Def:admissible}
An oriented graph $S$ is called \emph{admissible} if it is a union of $I = |\calV_Q|$ disjoint traps. A map $\psi: K \to S$ is a \emph{filling} of $(X,U)$ if:
\begin{enumerate}
\item $\psi$ takes vertices to vertices, and maps edges to edge paths. 
\item There is a bijection between $U_1, \dots , U_I$, the connected components of $U$,  and $S_1, \dots , S_I$, the connected components of $S$,  so that   \[ \rho(k) \in U_i \Leftrightarrow \psi(k) \in S_i \text{ for all } k \in K \] 
\item If $b,b'$ are veritcal edges in the same rectangle $R_\eps$ in $X_U$, then \[ len(\psi(b)) = len(\psi(b'))\] We denote this length by $ht_{\psi}(R_\eps)$
\end{enumerate}
\end{definition}

\noindent We glue $X_U$ to $S$ along $\psi:K \to S$. 

\begin{definition}\label{Def:Filling}
Let $X$ be a mapping torus of a map $f \from G \to G$, and let $U$ be the graph on the vertical edges.  Let $S$ be an admissible oriented graph and $\psi$ a filling of $(X,U)$. The complex obtained from $X$ by $\psi$-filling is denoted $X(U,S,\psi)$ and it is constructed from $X_U \cup_\psi S$ by adding $ht_\psi(R_\eps)-1$ parallel edges between $b_\eps$ and $b_{\bar\eps}$ in $R_\eps$ at preimages of vertices of $S$.
\end{definition}

\begin{proposition}\label{Prop:SurgeryInv}
Mapping torus surgery  is invertible, i.e.
\begin{enumerate}
\item $X(U,U,\rho) = X$ 
\item If $Z = X(U,S,\psi)$ then $X = Z(S,U,\rho)$.
\end{enumerate}
\end{proposition}
\begin{proof}
To prove (1) recall that $\xfp = X_U \cup_\rho U$. Since $X$ is obtained from $\xfp$ by subdividing each $c_e$ into $ht(c_e) = |\eps|$ cells we have $X = X(U,U,\rho)$. Item (2) follows from (1) and the fact that $X_U = \bigl( X(U,S,\psi) \bigr)_S$. 
\end{proof}

\begin{proposition}\label{Prop:NewMapTorus}
If $X,U,S,\psi$ are as above then $X(U,S,\psi)$ is a mapping torus. 
\end{proposition}

\begin{proof}
We need to check properties 1-4 in Proposition \ref{Prop:IdenMT}. 
The vertical edges of $X(U,S,\psi)$ are the edges of $S$, their complement will be the horizontal edges.  By assumption, $S$ is a union of traps so it satisfies (2) in \ref{Prop:IdenMT}.  To check (3), notice that there is a bijection from the set of cells of $X$ to the horizontal edges. After removing the non-dynamic edges, there is a bijection from the set of rectangles $\mathcal{R}$ to the edges of $\xfp$ not in $U$. This induces a bijection from the rectangles in $X_U$ to the edges outside $K$, which survives after gluing $K$ along $S$. So there is a bijection between the 2-cells of $X_U \cup_\psi S$ and the edges outside of $S$. Subdivision adds a new cell and a new top edge for it. So at the end of the subdivition process, there is a bijection $\top \from \calC \to \HR$. Property (4) of \ref{Prop:IdenMT} holds. This verifies that $X(U,S,\psi)$ is a mapping torus.
\end{proof}

\begin{proposition}\label{Prop:fQsurgery}
Given $X$ a mapping torus of $f \from G \to G$ and $U$ the graph on the vertical edges, let $S$ be a union of $|\calV_Q|$ disjoint 1-edge circles. Let $\psi$ be a filling where the height of every rectangle is 1. Then $X(U,S,\psi)$ is the mapping torus of $f_Q$.
\end{proposition}
\begin{proof}
This follows from Proposition \ref{Prop:SameFP}.
\end{proof}

\begin{example} 
Let us start with the map in Figure \ref{Fig:graph1}. 
\begin{figure}[ht]
\begin{center}
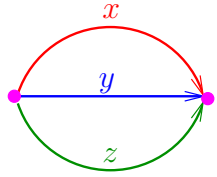
\caption{We construct the mapping\label{Fig:graph1} torus of the following map $f(x) = y$, $f(y)= z$ and $f(z) = yXz$.}
\end{center}
\end{figure}

This map is not a train-track map. It is its own archetype since every edge is dynamic. The structure of the mapping torus is given in Figure \ref{Fig:mtNew}. 
\begin{figure}[ht]
\begin{center}
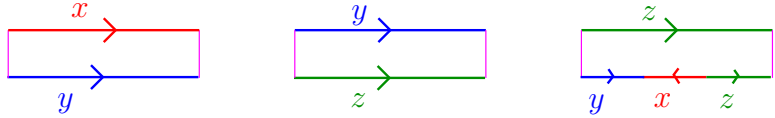
\caption{The mapping torus of $f$.}
\label{Fig:mtNew}
\end{center}
\end{figure}
The left pink edges are identified to the vertical circle $S_1$ and the right pink edges are identified to the vertical circle $S_2$. The vertices of $K$ correspond to the directed horizontal edges $x,y,z,X,Y,Z$ with the equivalence $Y\sim X$ and $x \sim z$. $K$ has edges $b_x = ([x],[y]), b_y = ([y],[z]), b_z = ([z],[y])$ that form the connected component $K_1$ and $b_X = ([X],[Y]), b_Y =  ([Y],[Z]), b_Z = ([Z],[Z])$ that form the connected component $K_2$. Figure \ref{Fig:KnS} shows the map $\rho: K \to S$. 
\begin{figure}[hb]
\begin{center}
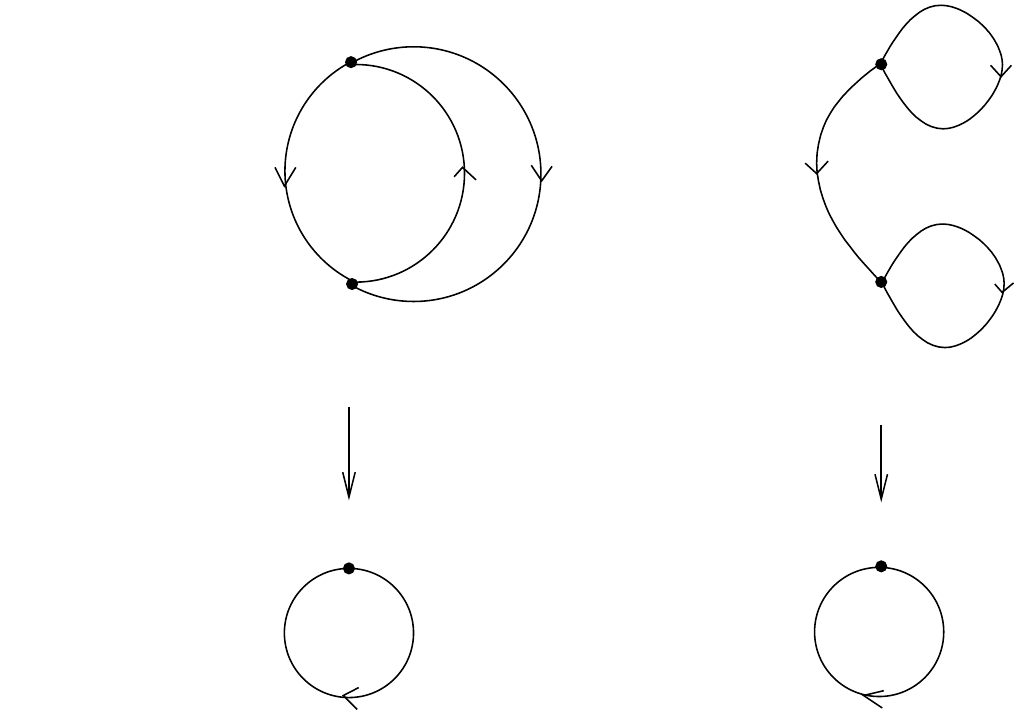
\caption{The map $\rho \from K \to S$.}
\label{Fig:KnS}
\end{center}
\end{figure}

\begin{figure}[ht]
\begin{center}
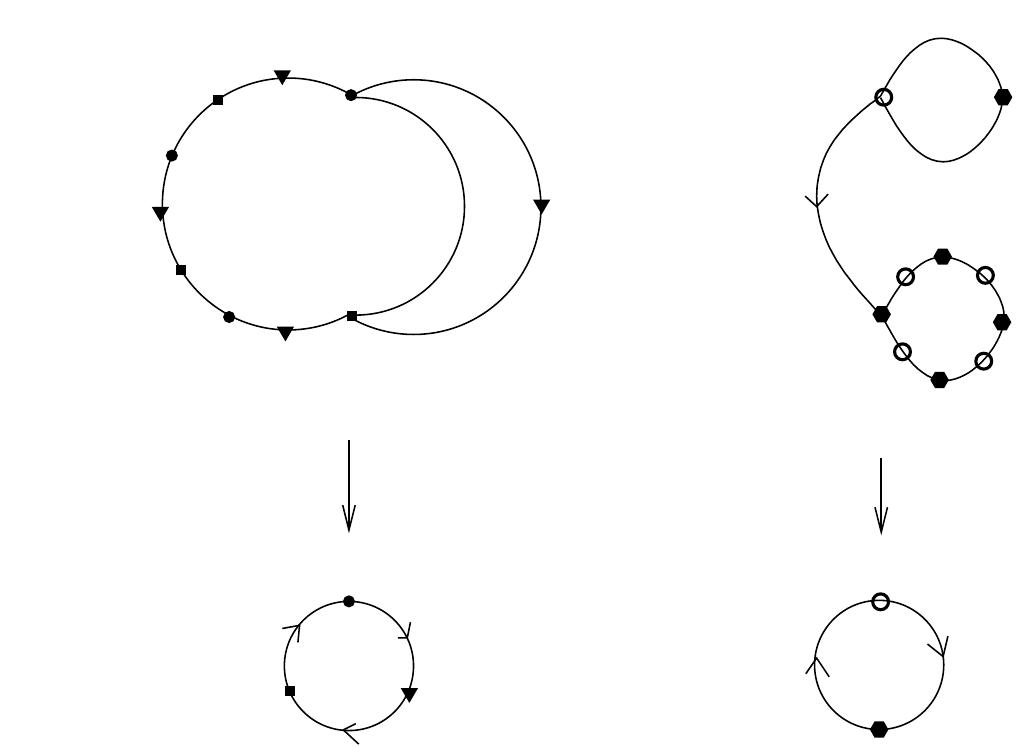
\caption{The map $\rho' \from K \to U$.}
\label{Fig:KnU}
\end{center}
\end{figure}

Changing $\rho$ to a map $\rho':K \to U$ as in Figure \ref{Fig:KnU}, and doing surgery gives a new mapping torus whose structure is described in Figure \ref{Fig:mt2}. The heights of the rectangles are $ht_\psi(R_x)=2, ht_\psi(R_y) = 1, ht_\psi(R_z) = 8$. The new map is $f'\from G' \to G'$ with $G'$ as in Figure \ref{Fig:graph2}. We have the map 
\[
x \to x_1 \to y \to z \to z_1 \to z_2 \to z_3 
\to z_4 \to z_5 \to  z_6 \to z_7 \to y_1X_1z_1.
\] 
Notice that $f'$ is an irreducible train-track map and its archetype is $f$.

\begin{figure}[ht]
\begin{center}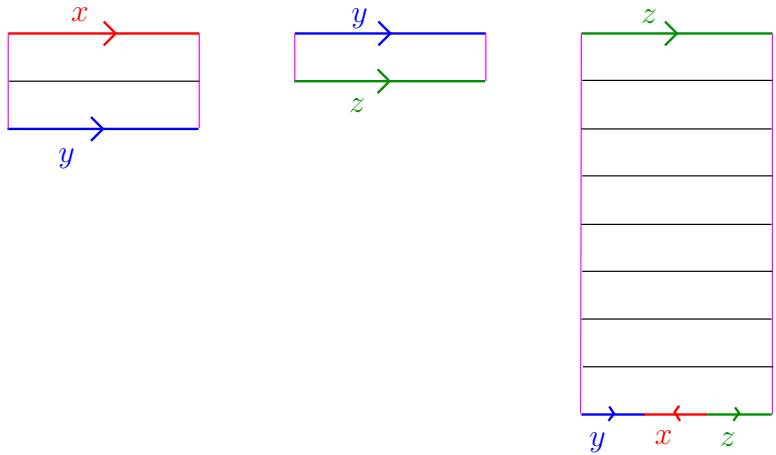
\caption{The new cell structure after surgery. Cells are subdivided according to their height. }
\label{Fig:mt2}
\end{center}
\end{figure}

\begin{figure}[ht]
\begin{center}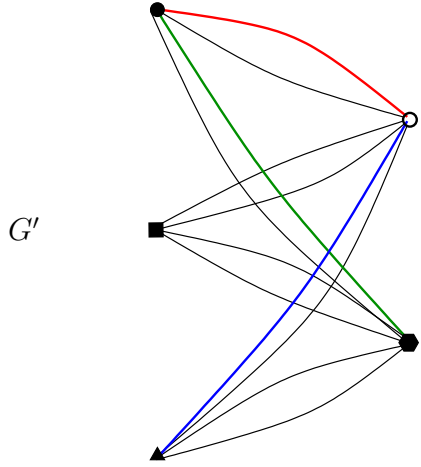
\caption{The complex obtained after surgery is the mapping torus of a map $f':G' \to G'$, with $G'$ as above. }
\label{Fig:graph2}
\end{center}
\end{figure}

\end{example}

\section{Finiteness} \label{Sec:finite}

We start by proving that if $f$ is $\sP$--small then the quotient graph $Q$ and the archetype of $f$ are uniformly finite. 

\begin{definition}
Let $\sB>0$ a map is $\sB$-bounded if the image of every edge is a concatenation of at most $\sB$ edges.
\end{definition}

\begin{proposition}\label{Prop:fQfinite1}
For every ${\sf P}$ there are positive integers ${\sf E}, \sB$ so that if $f$ is ${\sf P}$--small then  
\begin{enumerate}
\item $Q$ has at most $\sE$ edges and at most $\sE$ vertices. 
\item $f_Q$ is $\sB$--bounded.
\end{enumerate}
\end{proposition}
\begin{proof} For (1), it is enough to bound the number of edges in $Q$. 
This is a reiteration of the proof in \cite{FLM}. Let $f \from G \to G$ be an irreducible train-track map where $G$ has $m$ edges. Recall the transition matrix $M_{m \times m} = (a_{ij})$ of $f$. Define a graph $\Gamma$ where vertices correspond to edges of $G$ and there are $a_{ij}$ edges from $e_j$ to $e_i$. Define $d_{out}(v_i)$ as the number of edges coming out of $v_i$, i.e. the combinatorial length of $f(e_i)$ and $d_{in}(v)$ the edges coming into $v_i$ the total number of times $e_i$ appears in the image of an edge. We use a theorem of Ham and Song \cite{HS} that 
\[ 1+ \sum_{v \in \Gamma_f} \bigl( \dout(v)-1 \bigr) = 1+ \sum_{v \in \Gamma_f} \bigl( \din(v)-1 \bigr)  \leq \lam_f^n \] 
If $f$ is $\sP$ small then $\sum_{v \in \Gamma_f} \bigl( \dout(v)-1 \bigr) \leq \sP -1$. A vertex $v \in \Gamma_f$ has $\dout(v)>1$ iff $v$ corresponds to a mixed edge of $f$. Therefore, the number of mixed edges is $\leq \sP$ and if $e$ is mixed $f(e) \leq \sP$. If $e$ is an edge in $G$ such that there are distinct edges $e',e''$ in $G$ so that $f(e') = f(e'')=e$ then the vertex corresponding to $e$ in $\Gamma_f$ has $\din(v_e) >1$ so the number of such vertices is $\leq \sP$. The set of dynamic edges is made up of images of mixed edges and edges with more than one edge mapping to them. The number of dynamic edges is thus bounded by $\sP^2+\sP$. Therefore the number of edges in $Q< \sP^2 + \sP$. Moreover, the $f$-image of every edge is bounded by $\sP$. Thus $f_Q$ is $\sP$--bounded. 
\end{proof}

\begin{proposition}\label{Prop:fQfinite2}
For every $\sP$ there is a constant $\sA$ and maps $g_1, \dots , g_\sA$ such that $f_Q = g_i$ for some $i$ for any $\sP$--small itt map $f$.
\end{proposition}
\begin{proof}
There are only finitely many graphs with less than $\sE$ edges. There are only finitely many self maps of those graphs that are $\sB$-bounded. Therefore, by Proposition \ref{Prop:fQfinite1} there are only finitely many possibilities for $f_Q$. 
\end{proof}

\begin{theoremA} 
For every $\sP>0$, there is a finite set of $2$--complexes, which are mapping tori of self maps of graphs, $X_1, \ldots , X_\sA$, ,so that the following holds. If $f \from G \to G$ is a $\sP$--small  irreducible train-track map on a graph $G$, then $M_f$ is homeomorphic to a $2$--complex that is obtained by surgery on some $X_i$.   
\end{theoremA}
\begin{proof}
Let $g_1, \dots g_\sA$ be the maps from Proposition \ref{Prop:fQfinite2} define $X_i = M_{g_i}$. There is some $i$ such that $f_Q = g_i$. By Proposition \ref{Prop:fQsurgery} $X_i$ can be obtained from $M_f$ by surgery. By Proposition \ref{Prop:SurgeryInv} $M_f$ can be constructed  from $X_i$ by surgery. 
\end{proof}

\section{Bounded Presentations}

\begin{proposition}\label{Prop:finiteCW}
If $f \from G \to G$ is a $\sP$--small automorphism then $M_f$ is has a finite CW structure where the number of cells only depends on $\sP$.
\end{proposition}
\begin{proof}
Let $X = M_f$, and $U$ be the vertical graph, let $Y = M_{f_Q}$ and $S$ the vertical graph. By Proposition \ref{Prop:fQsurgery}, $X = Y(S,U,\rho)$, thus, the number of 2-cells in $X_U$ equals the number of 2-cells in $Y_S$ which is the number of 2-cells in $Y$. This number is bounded in terms of $\sP$ by Proposition \ref{Prop:fQfinite1}. \\

A priori, the number of vertices in $U$ might be large. We show that we can remove all but boundedly many. Recall that  $\rho: K \to S$ sends vertices to vertices and respects the edge orientations. A vertex of $S$ is a \emph{natural vertex} if it is the image of a vertex in $K$,  or it has valence $\neq 2$. All unnatural vertices are removable. \\

$U$ is a disjoint union of $|\calV| \leq \sE$  traps. Every trap is a union of a circle and trees $T_1, \dots , T_k$  where each $T_i$ is attached to the circle at a single point of $T_i$. 
We call a vertex in $T_i$ a high-valence vertex if its valence is $>2$. It is easy to verify (for example by induction on the number of vertices that,

\begin{equation}\label{Eq:Ver} \# \{ \text{high valence vertices in } T_i \} < \# \{ \text{valence } 1 \text{ vertices in } T_i \} -1 
\end{equation}
Therefore the number of all high valence vertices in $U$ is smaller than the number of valence 1 vertices. \\

We show that if $v \in S$ has valence 1 then $v$ is an image of a vertex in $K$. Consider $\rho \from K \to U$, if $\rho$ is not onto then there is a vertex in $X$ with no horizontal edge adjacent to it. Thus $G$ has an isolated vertex, but we assumed that all vertices in $G$ have valence at least 3. Therefore, $\rho$ must be onto. Let $v$ be a valence 1 vertex of $U$ and $u$ a point in $K$ mapping to $v$. If $u$ is not a vertex then it is contained in the interior of a side $b_\eps$. Let $u'$ be the top vertex of $b_\eps$, then the segment $[u',u]$ maps to a directed path the terminated at $v$. Since $\rho$ preserves orientation, $v$ is a terminal point of some edge, but it is also an initial point of an edge since $U$ is a trap. So $v$ cannot have valence 1. Therefore, if $v$ is a valence 1 vertex then it is an image of a vertex. 

Thus, the number of valence 1 vertices is no greater than the number of vertices of $K$. The vertices of $K$ correspond to equivalence classes of directed edges in $Q$. Thus, the number of valence 1 vertices is bounded by $2\sE$. By equation (\ref{Eq:Ver}) the number of natural vertices is smaller than $4\sE$. Removing the unnatural vertices of $S$, we get an honest CW structure on $X$ with at most $\sE$ 2-cells, at most $5\sE$ edges, and $4\sE$ vertices. \\
 
\end{proof}

\begin{theoremB}
There is a number $\sM$ depending only on $\sP$ so that if $\phi \in \out$ is is a $\sP$--small  itt automorphism then $\Gamma_\phi$ has a presentation with at most $\sM$ generators and $\sM$ relations.
\end{theoremB}
\begin{proof}
Consider $\widetilde{M_f}$ with the CW structure obtained in Proposition \ref{Prop:finiteCW}. The complex $\widetilde{M_f}$ is contractible, therefore, the presentation of the fundamental group may be read from the CW structure of $M_f$ whose number of cells is bounded by a quadratic  function of $\sP$. \end{proof}

\section{Examples} \label{Sec:Example}

\begin{example}\label{Exa:Rose}
Consider the rose with $n$ leaves $x_1, \ldots, x_n$ and the map $f$
defined by
\[
x_1 \to x_2 \to \ldots \to x_n \to x_1x_2. 
\]
This is an train-track map since it is positive. It is easy to verify that it is irreducible. The dilatation of $f$ is computed by declaring one of the edges, $x_1$, to have unit length and computing the other edge lengths by requiring that $f$ stretchs each edge by the factor $\lam$. Thus the lengths of $x_2, \dots , x_n$ are $\lam, \dots , \lam^{n-1}$ respectively and from $f(x_n) = x_1x_2$ we get the equation:
\begin{equation}\label{Eq:10}
\lam^n = 1+ \lam
\end{equation}
Let $t_n$ be the root of this equation. Clearly, 
$\displaystyle \lim_{n \to \infty}t_n = 1$. By taking log in equation (\ref{Eq:10}) we see that $\displaystyle \lim_{n \to \infty} n \log( t_n) = \log(2)$. \\
\end{example}

Define \[\underline{\lam}(n) = \inf \{ \lam_f \mid f \from G \to G \text{ with } rk(\pi_1 G ) =n \}\] The example above implies 
\[
\lim_{n \to \infty} n \log(\underline{\lam}(n) ) \leq \log(2).
\] 
In \secref{Lower} we provide a lower bound for the value of 
$\ulam(n)$.

\begin{question}
What are the asymptotics of $\underline{\lam}(n)$? In other words, what is the limit $\displaystyle \lim_{n \to \infty} n \log\underline{\lam}(n)$?
\end{question}

\begin{figure}[ht]
\begin{center}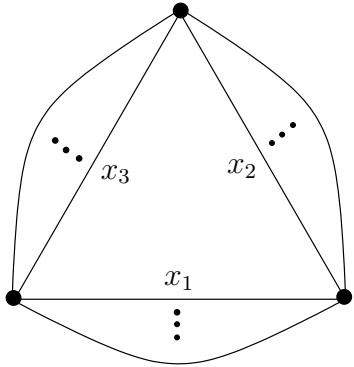
\caption{The rank of this graph is the number of edges subtracted by 2. Edges map homeomorphically $x_1 \to x_2 \to x_3 \to x_4 \to \dots \to x_{n+2}$. }
\label{Fig:tri}
\end{center}
\end{figure}

\begin{example} \label{Exa:Triangle}
We give a slightly better example than the rose example. Consider the graph in Figure \ref{Fig:tri}. The rank of this graph is the number of edges subtracted by 2. Edges map homeomorphically 
\[
x_1 \to x_2 \to x_3 \to x_4 \to \dots \to x_n \to x_{n+1} \to x_{n+2}.
\] 
When $n \md 3 = 0$ then $x_{n+2}$ is parallel to $x_3$ and we set $f(x_{n+2}) = X_3 X_2$. If $n \md 3 = 1$ then $x_{n+2}$ is parallel to $x_1$ and we set $f(x_{n+2}) = X_1 X_3$. If $n \md 3 = 2$ then $x_{n+2}$ is parallel to $x_2$ and we set $f(x_{n+2}) = X_2 X_1$. It is easy to verify that $f$ is an irreducible train-track map in all of these cases. The equation for the dilatation of $f$ is one of the three equations:
\[ \begin{array}{l}
	\lam^{n+2} = \lam+\lam^2 \\
	\lam^{n+2} = 1+ \lam^2 \\
	\lam^{n+2} = 1 + \lam
\end{array}
\] The roots of these are slightly smaller than of Equation (\ref{Eq:10}). Howvever, they still satisfy $\displaystyle \lim_{n \to \infty} n \log\underline{\lam}(n) = \log(2)$. 
\end{example}



\section{A lower-bound for dilatation} \label{Sec:Lower}

In this section we provide a lower-bound for $\ulam(n)$. This
does not match the upper-bound provided by examples in the previous
section. However, we arrive naturally at \exaref{Triangle} as, possibly, 
the itt with the smallest dilation. 

Let $f \from G \to G$ be a irreducible train-track map. Endow $G$
with the natural metric as in Definition \ref{Def:naturalMetric} so that $f$ is $\lambda_f$--Lipschitz. We scale $G$ so that the 
smallst edge $e$ has a length $1$. Let $\sE$ be the number of edges of $G$ and 
$\sV$ be the number vertices of $G$. 
Note that every edge of $G$ is in the image of $f^k(e)$ for some 
$k < \sE$. This is because each $f^k(e)$ has to cover at least one new 
edge that is not covered by $f(e), \ldots, f^{k-1}(e)$ until all edges are covered. 

This provides a quick lower-bound for $\lambda_f$. Namely, let $e_{mix}$ be 
a mixed edge. Then $f(e_{mix})$ contains at least two edges and hence has a 
length of at least 2. Assuming $e_{mix}$ is contained in $f^k(e)$ we have
$$
\lambda_f^{k+1} = |f^{k+1}(e)| \geq |f(e_{mix})| \geq 2
$$
Since $k +1 \leq \sE \leq 3n-3$, we have
$$
\log \lambda_f \geq \frac {\log2}{3n-3}.
$$

Using the same type of argument, with more care, we can 
replace $\log(2)$ with $\log(3)$ in the above estimate. 
If there are two mixed edges or if a mixed edge is
mapped to 3 edges, we have $|f^{k+1}(e)| \geq 3$. That is in 
these cases, 
$$
\log \lambda_f \geq \frac {\log3}{3n-3}.
$$
The remaining case is when there is one mixed edge $e_{mix}$ and it is
mapped to exactly two edges. That is,
$$
e, f(e), f^2(e), \ldots, f^{\sE-1}(e)=e_{mix},
$$
are all single edges and $f(e_{mix})$ contains $e$ and an adjacent edge.
Let $e_i = f^i(e)$. 

Further, we have every vertex of $G$ is an image of vertex. This is not always true when there is more than one mixed edge. However, in our case, every vertex is and endpoint of some edge $e_i$, $i\geq 1$, and $e_i$ is a homeomorphic image of $e_{i-1}$. Hence, $f$ acts by a 
permutation on the vertex set. We claim that there has to be only one vertex
orbit. Otherwise, we have at least two vertex orbits $A$ and $B$. 
There are $3$--types of edges, those connecting a vertex in $A$ to $A$, 
$A$ to $B$ or $B$ to $B$. The type of an edge is preserved unless
the edge is a mixed edge. That is, there a mixed edge for each type. 
Since we have only one mixed edge, there should be only
one type of edge, which is from $A$ to $B$ ($G$ is connected). But this 
implies that $G$ is a bi-partite graph and the image of the mixed edge has 
a length at least $3$ which has been dealt with before. 

We now analyze the case where there is only one vertex orbit. 
Label the vertices $0, 1, .. (\sV-1)$ with $f(i) = i+1$ and let $e$ be the 
edge $[0,k]$. Then all edges are in the form $[i, i+k]$, but there may be
more than one edge of type $[i, i+k]$. Since $G$ is connected, we have 
$$\text{gcd}(k,\sV)=1.$$
The last edge $[E, E+k]$ is mapped to a path of length two where one 
edge is $e=[0,k]$ and the other one an edge adjacent to $e$
(either $[k, 2k]$ or $[-k,0]$). That is, the end points of $[E, E+k]$
are mapped to either $[-k, k]$ or $[0,2k]$. Either way, we have
$$
 \sV = 3k
$$
The above two equations imply that $\sV=3$. That is, our graph $G$
and the map $f$ are exactly those describe in \exaref{Triangle}.
We have
$$
\lambda_f \sim \frac{\log 2}{n},
$$
which is larger than $\frac{\log 3}{3n-3}$. 

In fact, it seems that having more than one mixed edge should increase 
$\lambda_f$ even further and it is reasonable to conjecture that 
\exaref{Triangle} is the itt with lowest dilatation number. 

\bibliographystyle{alpha}
\bibliography{ref2}
\end{document}